\documentclass{amsart}
\usepackage{amsmath,amssymb,amsthm,comment,todonotes}
\usepackage{enumitem}
\theoremstyle{definition}
\newtheorem{definition}{Definition}[section]

\theoremstyle{plain}
\newtheorem{proposition}[definition]{Proposition}
\newtheorem{lemma}[definition]{Lemma}
\newtheorem{theorem}[definition]{Theorem}
\newtheorem{corollary}[definition]{Corollary}
\newtheorem*{theorem*}{Theorem}
\theoremstyle{remark}
\newtheorem{remark}[definition]{Remark}
\newcommand{\tr}{\textrm{tr}}
\newcommand{\compf}{\tilde{f}}

\title{Lie applicable surfaces and curved flats}

\author{Francis Burstall}
\address{(F. Burstall) Department of Mathematical Sciences, University of Bath, Bath. BA2 7AY. United Kingdom}
\email{feb@bath.ac.uk} 

\author{Mason Pember}
\address{(M. Pember) Dipartimento di Matematica, Politecnico di Torino,
Corso Duca degli Abruzzi 24, I-10129 Torino, Italy}
\email{mason.pember@polito.it}

\begin{document}
\maketitle

\begin{abstract}
We investigate curved flats in Lie sphere geometry. We show that in this setting curved flats are in one-to-one correspondence with pairs of Demoulin families of Lie applicable surfaces related by Darboux transformation. 
\end{abstract}

\section{Introduction}
Curved flats, introduced by Ferus-Pedit~\cite{FP1996}, are an integrable system for maps into symmetric spaces. In~\cite{BHPP1997} it is shown how Darboux pairs of isothermic surfaces are equivalent to curved flats in the symmetric space of point pairs in the conformal $3$-sphere. In~\cite{B1968} curved flats in the space of circles (called hypercyclic systems) of $S^{3}$ are studied. It is shown that the orthogonal surfaces to these are Guichard surfaces that are mutually Eisenhart transforms of each other with parameter $\mu = 0$. This was extended to $S^{n}$ in~\cite{BCwip,H1996}. Curved flats in Laguerre geometry were studied in~\cite{MN2000}. These are constructed from the Gauss maps of Bianchi-Darboux pairs of $L$-isothermic surfaces. 

In~\cite{D1911iii,D1911i,D1911ii} Demoulin discovered $\Omega$-surfaces and $\Omega_{0}$-surfaces. Blaschke~\cite{B1929} showed that together these surfaces are the applicable surfaces of Lie sphere geometry. In~\cite{F2000ii,MN2006} it is shown that these surfaces constitute an integrable system and this is given a gauge theoretic interpretation in~\cite{C2012i,P2018}. This gives rise to a Lie-geometric analogue of the Darboux transformation for these surfaces. In this paper we show how curved flats in Lie sphere geometry are equivalent to a pair of Demoulin families of Lie applicable surfaces such that any two members of opposite families form a Darboux pair. 

\textit{Acknowledgements.} 
This work was supported by the Austrian Science Fund (FWF) through the research project P28427-N35 ``Non-rigidity and symmetry breaking", GNSAGA of INdAM and the MIUR grant ``Dipartimenti di Eccellenza'' 2018 - 2022, CUP: E11G18000350001, DISMA, Politecnico di Torino.

\section{Lie sphere geometry}
\label{sec:lsg}

Let $n\in \mathbb{N}$ with $n\ge 2$ and let $\mathbb{R}^{n+2,2}$ denote a $(n+4)$-dimensional vector space equipped with a non-degenerate symmetric bilinear form $(.,.)$ of signature $(n+2,2)$. Let $\mathcal{L}$ denote the lightcone of $\mathbb{R}^{n+2,2}$. In Lie sphere geometry, the projective lightcone $\mathbb{P}(\mathcal{L})$ parametrises oriented spheres of $n$-dimensional space forms so that oriented spheres are in contact if and only if the line joining them lies in $\mathbb{P}(\mathcal{L})$ (see~\cite{C2008}). In this way the set of lines in $\mathbb{P}(\mathcal{L})$, denoted by $\mathcal{Z}$, is identified with the manifold of contact elements. 

It is well known that the exterior algebra $\wedge^{2}\mathbb{R}^{n+2,2}$ is isomorphic to the Lie algebra $\mathfrak{o}(n+2,2)$ of $\textrm{O}(n+2,2)$, i.e., the space of skew-symmetric endomorphisms of $\mathbb{R}^{n+2,2}$, via the isomorphism
\[ a\wedge b\mapsto (a\wedge b),\]
where for any $c\in\mathbb{R}^{n+2,2}$, 
\[ (a\wedge b)c = (a,c)b - (b,c)a.\]
We shall make silent use of this identification throughout this paper. 

Suppose that $\Sigma$ is a $n$-dimensional manifold and consider a smooth map $f:\Sigma\to \mathcal{Z}$. Such a map can be equivalently considered as a rank 2 null subbundle of the trivial bundle $\underline{\mathbb{R}}^{n+2,2}:=\Sigma\times\mathbb{R}^{n+2,2}$. In Lie sphere geometry, one studies hypersurfaces in space forms by considering their contact lifts. These amount to smooth maps $f:\Sigma\to \mathcal{Z}$ satisfying $f^{(1)} \le f^{\perp}$, where $f^{(1)}$ denotes the set of sections of $f$ and derivatives of sections of $f$. We call such an $f$ a \textit{Legendre map}\footnote{Note that this a weaker condition than that of~\cite{C2008}, as we do not require the ``immersion condition''.}.

\subsection{Splitting of the trivial bundle}
Suppose that $\Sigma$ is an $n$-dimensional manifold and that $f,\compf:\Sigma\to\mathcal{Z}$ are complementary smooth maps into $\mathcal{Z}$, that is, we may view them as rank 2 null subbundles of $\underline{\mathbb{R}}^{n+2,2}$ such that $f\cap \compf = \{0\}$. We then have that $f\oplus\compf$ is a rank 4 bundle of $(2,2)$ planes. We may then split the trivial bundle as 
\[ \underline{\mathbb{R}}^{n+2,2} = f\oplus \compf\oplus U,\]
where $U:= (f\oplus \compf)^{\perp}$. The trivial connection on $\underline{\mathbb{R}}^{n+2,2}$ can then be expressed as  
\[ d = D+  \mathcal{N}_{f} +\mathcal{N}_{\compf} + A_{f,\compf} + A_{\compf,f},\]
where $D$ is a metric connection preserving $f$, $\compf$ and $U$, and $\mathcal{N}_{f}\in \Omega^{1}(\compf\wedge U)$, $\mathcal{N}_{\compf}\in \Omega^{1}(f\wedge U)$, $A_{f,\compf}\in \Omega^{1}(\wedge^{2}\compf)$ and $A_{\compf,f}\in \Omega^{1}(\wedge^{2} f)$. 

Let $\sigma\in \Gamma f$. Then
\[ d\sigma = D\sigma + \mathcal{N}_{f}\sigma + A_{f,\compf}\sigma\]
with $D\sigma \in \Omega^{1}(f)$, $\mathcal{N}_{f}\sigma \in \Omega^{1}(U)$ and $A_{f,\compf}\sigma\in \Omega^{1}(\compf)$. Since $f^{\perp} = f\oplus U$, one deduces the following:

\begin{lemma}
\label{lem:legA}
$f$ is a Legendre map if and only if $A_{f,\compf}=0$. 
\end{lemma}

Since $d$ is the trivial connection, and is thus flat, we obtain the Gauss-Codazzi-Ricci equations for this splitting:
\begin{align}
R^{D}+[\mathcal{N}_{f}\wedge \mathcal{N}_{\compf}]+[A_{f,\compf}\wedge A_{\compf,f}] &=0, \label{eqn:RD}\\
d^{D}\mathcal{N}_{f} + [\mathcal{N}_{\compf}\wedge A_{f,\compf}] &=0,\label{eqn:DNf}\\
d^{D}\mathcal{N}_{\compf}+[\mathcal{N}_{f}\wedge A_{\compf,f}] &=0,\label{eqn:DNhatf}\\
d^{D}A_{f,\compf} + \frac{1}{2}[\mathcal{N}_{f}\wedge \mathcal{N}_{f}] &= 0, \label{eqn:DA}\\
d^{D}A_{\compf,f} + \frac{1}{2}[\mathcal{N}_{\compf}\wedge\mathcal{N}_{\compf}] &= 0 \label{eqn:dhatA}.
\end{align}

\subsection{Ribaucour transforms}

Classically two submanifolds are Ribaucour transforms of each other if they envelope a sphere congruence in such a way that the curvature directions at corresponding points coincide. In~\cite{BH2006} this was given a Lie geometric formulation which we shall recall now. 

Suppose that $f,\hat{f}:\Sigma\to\mathcal{Z}$ are Legendre maps that envelope a common sphere congruence $s_{0}:\Sigma\to \mathbb{P}(\mathcal{L})$, i.e., $f\cap \hat{f} = s_{0}$.  Consider the bundle 
\[ \mathcal{H}_{f,\hat{f}}:= (f+\hat{f})/s_{0}.\] 
This is a rank 2 subbundle of $s_{0}^{\perp}/s_{0}$ inheriting a metric of signature $(1,1)$. One then obtains a well defined orthogonal projection $\pi: s_{0}^{\perp}/s_{0} \to  \mathcal{H}_{f,\hat{f}}$ and using this one defines a metric connection on $\mathcal{H}_{f,\hat{f}}$ 
\begin{equation} 
\label{eqn:nabla}
\nabla^{f,\hat{f}}(\nu + s_{0}) = \pi(d\nu +s_{0})
\end{equation}
where $\nu \in \Gamma(f+\hat{f})$. 

\begin{definition}[\cite{BH2006}]
$f$ and $\hat{f}$ are \textit{Ribaucour transforms} of each other if $\nabla^{f,\hat{f}}$ is flat. We call $s_{0}$ a \textit{Ribaucour sphere congruence} and say that $(f,\hat{f})$ is a \textit{Ribaucour pair}.
\end{definition}

\subsection{Lie applicable surfaces}

\begin{definition}[\cite{P2018}]
\label{def:lieapp}
A Legendre map $f:\Sigma\to\mathcal{Z}$ is \textit{Lie applicable} if there exists $\eta\in \Omega^{1}(f\wedge f^{\perp})$ such that $d\eta = [\eta\wedge \eta]=0$ and the quadratic differential $q$ defined by 
\begin{equation} 
\label{eqn:etaquad}
q(X,Y) = \tr(f\to f: \sigma\mapsto \eta(X)d_{Y}\sigma)
\end{equation}
is non-zero on a dense open set. We call such an $\eta$ a \textit{gauge potential} of $f$. 
\end{definition}

\begin{remark}
Gauge potentials $\eta$ are not unique. Given any $\tau\in \Gamma( \wedge^{2}f)$, one has that $\tilde{\eta} := \eta - d\tau$ is again a closed 1-form with $[\tilde{\eta}\wedge \tilde{\eta}]=0$ and $\tilde{q} = q$. We say that $\tilde{\eta}$ is \textit{gauge equivalent} to $\eta$. Let $[\eta]$ denote the equivalence class of gauge potentials that are gauge equivalent to $\eta$. 
\end{remark}

\begin{remark}
Clearly scales $\lambda \eta$ of gauge potentials $\eta$ by a constant $\lambda\in\mathbb{R}\backslash\{0\}$ are again gauge potentials. Some Legendre maps, such as those with planar curvature lines (see~\cite{MN2006}), can also be multiply Lie applicable in the sense that they admit multiple gauge potentials that are not related by scale or gauge equivalence. To avoid this ambiguity, when we refer to a Lie applicable Legendre map $f$, we shall be referring to the pair $(f,[\eta])$, for a fixed choice of $[\eta]$.
\end{remark}

Let $f$ be a Lie applicable surface with gauge potential $\eta$. Then one obtains a 1-parameter family of flat connections $\{d+t\eta\}_{t\in\mathbb{R}}$ on the trivial bundle. Parallel subbundles of these flat connections give rise to new Lie applicable surfaces in the following way: suppose that $\hat{s}$ is a parallel rank 1 null subbundle of $d+m\eta$ for some $m\in \mathbb{R}\backslash\{0\}$ that lies nowhere in $f$. Let $s_{0}:= f\cap \hat{s}^{\perp}$ and define $\hat{f} := s_{0}\oplus \hat{s}$. 

\begin{definition}
We call $\hat{f}$ an $m$-\textit{Darboux transform} of $f$.
\end{definition}

In~\cite{C2012i,P2018} it is shown that $\hat{f}$ is also a Lie applicable surface and that $f$ is an $m$-Darboux transform of $\hat{f}$. Thus we say that $(f,\hat{f})$ is an \textit{$m$-Darboux pair}. 

Since $\hat{s}$ is parallel subbundle of $d+m\eta$, there exists a section $\hat{\sigma}\in \Gamma \hat{s}$ such that $(d + m\eta)\hat{\sigma} =0$. Thus $d\hat{\sigma} \in \Omega^{1}((f+\hat{f})^{\perp})$. It then follows that the connection defined in~\eqref{eqn:nabla} is flat. Thus $m$-Darboux pairs are Ribaucour pairs.

\section{Curved flats in $G_{2,2}(\mathbb{R}^{n+2,2})$}
Let $G_{2,2}(\mathbb{R}^{n+2,2})$ denote the Grassmannian of 4-dimensional linear subspaces of $\mathbb{R}^{n+2,2}$ with signature $(2,2)$. Given an $n$-dimensional manifold $\Sigma$, any smooth map $W:\Sigma\to G_{2,2}(\mathbb{R}^{n+2,2})$ may be viewed as a rank 4 subbundle of the trivial bundle $\underline{\mathbb{R}}^{n+2,2}$. We then have a splitting
\[ \underline{\mathbb{R}}^{n+2,2} = W\oplus W^{\perp},\]
This yields a splitting of the trivial connection as
\[ d = \mathcal{D}^{W} + \mathcal{N}^{W},\]
where $\mathcal{D}^{W}$ is the sum of the induced connections on $W$ and $W^{\perp}$, and $\mathcal{N}^{W} \in \Omega^{1}(W\wedge W^{\perp})$. 

\begin{definition}
We say that $W:\Sigma\to G_{2,2}(\mathbb{R}^{n+2,2})$ is \textit{regular} if the metric given by\footnote{Note that this metric is a scale of the metric induced by the Killing form on $ G_{2,2}(\mathbb{R}^{n+2,2})$.} $tr(\mathcal{N}^{W}\circ \mathcal{N}^{W}|_{W})$ is non-zero on a dense open set. 
\end{definition}

If $\mathcal{D}^{W}|_{W}$ is flat and $\Sigma$ is simply connected then there are two families of $D^{W}$-parallel null rank 2 subbundles $\{f_{\alpha}\}, \{\hat{f}_{\beta}\}$ of $W$, each parametrised by a circle, i.e., $\alpha, \beta\in \mathbb{R}P^{1}$. In~\cite{BH2006} it is shown that each $f_{\alpha}, \hat{f}_{\beta}$ is a Legendre map and any pair $(f_{\alpha}, \hat{f}_{\beta})$ is a Ribaucour pair. We call $\{f_{\alpha}\}, \{\hat{f}_{\beta}\}$ the \textit{Demoulin families of $W$.} This is a Lie geometric reformulation of the Bianchi Permutability Theorem for Ribaucour transforms~\cite[\S 354]{B1923}.

\begin{remark}
\label{rem:inter}
The intersection points of the Demoulin families $s_{\alpha,\beta} = f_{\alpha}\cap \hat{f}_{\beta}$ are the parallel null rank 1 subbundles of $\mathcal{D}^{W}|_{W}$, while distinct members of the same family are complementary. 
\end{remark}

In this paper we shall be considering the case where in addition to $\mathcal{D}^{W}|_{W}$ being flat, we have that $\mathcal{D}^{W}|_{W^{\perp}}$ is flat as well, i.e., $\mathcal{D}^{W}$ is flat. This brings us to the following notion: 

\begin{definition}[\cite{FP1996}]
A map $W:\Sigma\to G_{2,2}(\mathbb{R}^{n+2,2})$ is called a \textit{curved flat} if $\mathcal{D}^{W}$ is flat. 
\end{definition}

The flatness of the trivial connections yields the equations: 
\begin{align}
R^{\mathcal{D}^{W}} + \frac{1}{2}[\mathcal{N}^{W}\wedge\mathcal{N}^{W}]&=0, \label{eqn:curvRDW}\\
d^{\mathcal{D}^{W}}\mathcal{N}^{W}&=0 \label{eqn:codDW}.
\end{align}
Now suppose that $f,\compf\le W$ are complementary null line subbundles of $W$. Then we can write things in terms of the splitting from Section~\ref{sec:lsg}: 
\begin{align*}
\mathcal{D}^{W} &= D + A_{f,\compf} + A_{\compf,f}\\
\mathcal{N}^{W} &= \mathcal{N}_{f}+ \mathcal{N}_{\compf}.
\end{align*}
From~\eqref{eqn:curvRDW}, we see that $\mathcal{D}^{W}$ is flat if and only if 
\[ [\mathcal{N}^{W}\wedge \mathcal{N}^{W}] = [\mathcal{N}_{f}\wedge \mathcal{N}_{f}] + 2 [\mathcal{N}_{f}\wedge \mathcal{N}_{\compf}] + [\mathcal{N}_{\compf}\wedge \mathcal{N}_{\compf}]=0.\]
Since the components of this sum live in linearly independent subbundles, we have that $\mathcal{D}^{W}$ is flat if and only if 
\[ [\mathcal{N}_{f}\wedge \mathcal{N}_{f}] = [\mathcal{N}_{f}\wedge \mathcal{N}_{\compf}] = [\mathcal{N}_{\compf}\wedge \mathcal{N}_{\compf}]=0.\]

Suppose now that $f$ is a Legendre map. By Lemma~\ref{lem:legA}, we have that $A_{f,\compf}=0$. Then from~\eqref{eqn:DA}, we learn that $[\mathcal{N}_{f}\wedge \mathcal{N}_{f}]=0$. Define $\eta :=- \frac{1}{m}(A_{\compf,f}+\mathcal{N}_{\compf}) \in \Omega^{1}(f\wedge f^{\perp})$, for some $m\in \mathbb{R}\backslash\{0\}$. Then
\begin{align*}
d\eta &= -\frac{1}{m} (D + A_{\compf,f}+\mathcal{N}_{f} + \mathcal{N}_{\compf})(A_{\compf,f}+\mathcal{N}_{\compf})\\
& = -\frac{1}{m}\left(d^{D}A_{\compf,f} + d^{D}\mathcal{N}_{\compf} + [A_{\compf,f}\wedge \mathcal{N}_{\compf}] + [\mathcal{N}_{f}\wedge A_{\compf,f}] + [\mathcal{N}_{f}\wedge \mathcal{N}_{\compf}] + [\mathcal{N}_{\compf}\wedge \mathcal{N}_{\compf}]\right)\\
&= -\frac{1}{m}\left(\frac{1}{2}[\mathcal{N}_{\compf}\wedge\mathcal{N}_{\compf}] + [\mathcal{N}_{f}\wedge \mathcal{N}_{\compf}]\right),
\end{align*}
using~\eqref{eqn:DNhatf} and~\eqref{eqn:dhatA}. It then follows that $d\eta=0$ if and only if 
\[ [\mathcal{N}_{f}\wedge \mathcal{N}_{\compf}] = [\mathcal{N}_{\compf}\wedge \mathcal{N}_{\compf}]=0.\]
On the other hand, we already have that $[\mathcal{N}_{f}\wedge \mathcal{N}_{f}]=0$ since $f$ is a Legendre map. Thus $\eta$ is closed if and only if $\mathcal{D}^{W}$ is flat. We also have that $[\eta\wedge \eta] = - \frac{1}{m^{2}}[\mathcal{N}_{\compf}\wedge \mathcal{N}_{\compf}]$. Therefore if $\eta$ is closed then $[\eta\wedge \eta]=0$. We have thus arrived at the following lemma: 

\begin{lemma}
\label{lem:flatclosed}
Suppose that $W=f\oplus\compf$ and $f$ is a Legendre map. Then $\mathcal{D}^{W}$ is flat if and only if the 1-form $\eta = -\frac{1}{m}(A_{\compf,f} +\mathcal{N}_{\compf})\in \Omega^{1}(f\wedge f^{\perp})$ is closed for any $m\in \mathbb{R}\backslash\{0\}$. Moreover, if $\eta$ is closed then $[\eta\wedge \eta]=0$. 
\end{lemma}

We are almost in the position to say that $f$ is a Lie applicable surface, but we have to determine whether the quadratic differential $q$ defined by~\eqref{eqn:etaquad} is non-zero on a dense open set. Now for $\sigma\in \Gamma f$, 
\begin{align*}
 \eta(X)d_{Y}\sigma &= -\frac{1}{m}(A_{\compf,f}(X)+\mathcal{N}_{\compf}(X))(D_{Y}\sigma + \mathcal{N}_{f}(Y)\sigma)\\
 &= -\frac{1}{m}\mathcal{N}_{\compf}(X)\mathcal{N}_{f}(Y)\sigma\\
 &= - \frac{1}{m}\mathcal{N}^{W}(X)\mathcal{N}^{W}(Y)\sigma.
\end{align*}
Therefore, the quadratic differential associated to $\eta$ is given by $q =-\frac{1}{m} \tr(\mathcal{N}^{W}\circ \mathcal{N}^{W}|_{f})$. Moreover, since $\mathcal{N}^{W}$ is skew-symmetric, we can in fact say that $q =-\frac{1}{2m} \tr(\mathcal{N}^{W}\circ \mathcal{N}^{W}|_{W})$. Hence $q$ is non-zero on a dense open set if and only if $W$ is regular. We also have that
\[ d+ m\eta = D+ \mathcal{N}_{f},\]
and thus $\compf$ is a parallel subbundle of $d+m\eta$. 

\begin{proposition}
\label{prop:flatpar}
Suppose that $W$ is a regular curved flat. Then any Legendre map $f\le W$ is Lie applicable and any complementary $\compf\le W$ to $f$ is a parallel subbundle of $d+m\eta$ for some gauge potential $\eta\in \Omega^{1}(f\wedge f^{\perp})$ and $m\in \mathbb{R}\backslash \{0\}$.  
\end{proposition}

Conversely, let $f$ be a Lie applicable Legendre map with gauge potential $\eta$ and suppose that $\compf:\Sigma\to \mathcal{Z}$ is a complementary parallel subbundle of the flat connection $d+m\eta$. Let $W:=f\oplus\compf$. Since $f^{\perp} = f\oplus W^{\perp}$, we may write $\eta = \eta_{\wedge^{2}f} + \eta_{f\wedge W^{\perp}}$ with $\eta_{\wedge^{2}f} \in \Omega^{1}(\wedge^{2}f)$ and $\eta_{f\wedge W^{\perp}}\in \Omega^{1}(f\wedge W^{\perp})$. Now let us consider the connection $d+m\eta$: 
\[ d+ m\eta = D +\mathcal{N}_{f}+ (A_{\compf,f} + m\eta_{\wedge^{2}f}) + (\mathcal{N}_{\compf}+ m\eta_{f\wedge W^{\perp}}).\]
One deduces that $\compf$ is a parallel subbundle of $d+m\eta$ if and only if $A_{\compf,f}=- m\eta_{\wedge^{2}f}$ and $\mathcal{N}_{\compf}=- m\eta_{f\wedge W^{\perp}}$. Hence, $\eta = - \frac{1}{m}(A_{\compf,f} + \mathcal{N}_{\compf})$. Now, by Lemma~\ref{lem:flatclosed}, the closure of $\eta$ implies that $\mathcal{D}^{W}$ is flat. 

\begin{proposition}
\label{prop:parflat}
Suppose that $f$ is a Lie applicable Legendre map with gauge potential $\eta$ and suppose that a complementary $\compf$ is $d+m\eta$ parallel for some $m\in \mathbb{R}\backslash\{0\}$. Then $W:= f\oplus \compf$ is a regular curved flat. 
\end{proposition}

\begin{remark}
Note that any other complementary $\compf'\le W$ can be written as $\compf' = \exp(m\tau) \compf$ for some $\tau\in \Gamma(\wedge^{2}f)$. On the other hand $\exp(m\tau)\compf$ is a parallel subbundle of $d+m\eta'$ where $\eta':= \eta - d\tau$ is gauge equivalent to $\eta$. In particular, we learn that the $m$ in Proposition~\ref{prop:flatpar} is independent of the choice of complementary $\compf$. 
\end{remark}

Since $d+t\eta$ is a flat connection for any $t\in\mathbb{R}$, it admits many parallel null subbundles $\compf$. Proposition~\ref{prop:flatpar} and Proposition~\ref{prop:parflat} then yield:

\begin{corollary}
A Legendre map $f$ is Lie applicable if and only if there exists a regular curved flat $W:\Sigma\to G_{2,2}(\mathbb{R}^{n+2,2})$ with $W^{\perp}\le f^{\perp}$. 
\end{corollary}

We now seek to understand the geometry of this situation. Recall that $W$ such that $\mathcal{D}^{W}|_{W}$ is flat contain two families of Legendre maps $\{f_{\alpha}\}$, $\{\hat{f}_{\beta}\}$, with $\alpha,\beta\in \mathbb{R}P^{1}$. We shall now see how the additional condition that $\mathcal{D}^{W}|_{W^{\perp}}$ is flat affects these Demoulin families. 

\begin{theorem}
\label{thm:curvflat}
Suppose that $W:\Sigma\to G_{2,2}(\mathbb{R}^{n+2,2})$ is regular with $\mathcal{D}^{W}|_{W}$ flat. Then $W$ is a curved flat if and only if the Demoulin families of $W$ consist of Lie applicable Legendre maps $f_{\alpha},\hat{f}_{\beta}$ such that for some $m\in\mathbb{R}\backslash\{0\}$, $(f_{\alpha},\hat{f}_{\beta})$ are $m$-Darboux pairs for all $\alpha,\beta\in \mathbb{R}P^{1}$.
\end{theorem}

\begin{proof}
Suppose that $W$ is a curved flat and fix $f, \tilde{f}\in\{f_{\alpha}\}$. By Lemma~\ref{lem:legA}, Lemma~\ref{lem:flatclosed} and Proposition~\ref{prop:flatpar}, we have that $f$ is Lie applicable with associated 1-form $\eta = - \frac{1}{m} \mathcal{N}_{\compf}$ for some $m\in \mathbb{R}\backslash\{0\}$. 
Fix $\hat{f}\in \{\hat{f}_{\beta}\}$. Then by Remark~\ref{rem:inter}, $\hat{s}:= \compf\cap \hat{f}$ is parallel for $\mathcal{D}^{W} = D$. Thus $\hat{s}$ is parallel for $d+m\eta = D + \mathcal{N}_{f}$. Hence $\hat{f}$ is an $m$-Darboux transform of $f$. 

Conversely, let $f\in \{f_{\alpha}\}$ be a regular Lie applicable Legendre map and let $\hat{f}_{1},\hat{f}_{2},\hat{f}_{3}\in\{\hat{f}_{\beta}\}$ be $m$-Darboux transforms of $f$. Hence there exist rank 1 subbundles $\hat{s}_{i}\le \hat{f}_{i}$ that are parallel for $d+m\eta$, for some gauge potential $\eta$ of $f$. Let $\compf:= \hat{s}_{1}+\hat{s}_{2} + \hat{s}_{3}$. If $\tilde{f}$ has rank 2 then Proposition~\ref{prop:parflat} implies that $W = f\oplus \tilde{f}$ is a curved flat. Otherwise, suppose that $\compf$ has rank 3 on some open subset $U$ of $\Sigma$. Without loss of generality, assume that $U=\Sigma$. Then $s:= \compf\cap f$ is a rank 1 subbundle of $f$. Since $\tilde{f}$ is parallel for $d+m\eta$ we have that $d\sigma = (d+m\eta)\sigma \in \Omega^{1}( \tilde{f})$ for any $\sigma \in \Gamma s$. On the other hand, since $f$ is a Legendre map, we have that $d\sigma \in \Omega^{1}(f^{\perp})$. Thus, $d\sigma \in \Omega^{1}(s)$ and $s$ is constant. It then follows that there exists $\hat{f}\in\{\hat{f}_{\beta}\}$ with $f\cap \hat{f} = s$. Now $\hat{f}$ is an $m$-Darboux transform of $f$ and thus there exists a $d+m\eta$ parallel subbundle $\hat{s}\le \hat{f}$ with $\hat{s}\cap f=\{0\}$. Since we already have that $s\le \hat{f}$ is $d+m\eta$ parallel, it follows that $\hat{f}$ is parallel for $d+m\eta$. Assuming that $\hat{f}\neq \hat{f}_{1}$ (otherwise apply an analogous argument with $\hat{s}_{2}\le \hat{f}_{2}$) we may now choose $\hat{s}'\le \hat{f}$ that is $d+m\eta$ parallel with $\hat{s}'\perp \hat{s}_{1}$. Setting $\tilde{f}'= \hat{s}'\oplus \hat{s}_{1}$ yields a rank 2 null subbundle of $W$ complementary to $f$ that is parallel for $d+m\eta$. Applying Proposition~\ref{prop:parflat} then yields the result. 
\end{proof}

\noindent\textit{Examples:}
\begin{enumerate}[leftmargin=*, label=\arabic*.]
\item If $W$ is a curved flat containing a constant timelike vector $\mathfrak{p}$ then $W^{\perp}$ is a curved flat in the space of circles (a hypercyclic system) of the conformal geometry of $\mathfrak{p}^{\perp}\cong \mathbb{R}^{n+2,1}$ . The Demoulin families of $W$ are then the contact lifts of the orthogonal submanifolds of this hypercyclic system. These are known to be conformally flat hypersurfaces that are mutually Eisenhart transforms with parameter $\mu=0$. See~\cite{B1968, BCwip,H1996}. 
\item In the case that $n=2$ and $W$ is a curved flat containing a constant lightlike vector $\mathfrak{q}$, $W\cap  \mathfrak{q}^{\perp}$ projects to a curved flat in the Laguerre geometry $\mathfrak{q}^{\perp}/\mathfrak{q}\cong \mathbb{R}^{3,1}$. The Demoulin families of $W$ are then contact lifts of $L$-isothermic surfaces that are mutually Bianchi-Darboux transforms of each other. Geometrically the curved flat represents the map into the space of point pairs of $S^{2}$ built from the Gauss maps of such Bianchi-Darboux transforms. See~\cite{MN2000}. 
\end{enumerate}

\bibliographystyle{plain}
\bibliography{bibliographycf}

\end{document}